\documentclass[11pt, a4paper]{article}
	
\usepackage{amsmath, amsthm, amscd, amsfonts, amssymb, graphicx, color}
\usepackage[bookmarksnumbered, plainpages]{hyperref}

 \newtheorem{thm}{Theorem}

\begin{document}
\centerline{\Large{\bf The abstract Cauchy problem for the non-}}
\centerline{}
\centerline{\Large{\bf stationary bulk queue $M(t) \vert M[k, B] \vert 1$}}
 \centerline{}
\centerline{Yong-Chol Chon}
 \centerline{}
 \small \centerline{Faculty of Mathematics, \textbf{Kim Il Sung} University, D.P.R Korea}
\small \centerline{e-mail address: chonyongchol@yahoo.com}
\centerline{}
\centerline{}
\begin{abstract}
We derived state probability equations describing the queue $M(t) \vert$ $M[k, B] \vert 1$  and formulated as an abstract Cauchy problem to investigate 
by means of the semi-group theory of bounded linear operators in functional analysis. 
With regard to the abstract Cauchy problem of this queue, we determined the eigenfunctions of the maximal operator and showed some properties 
of the Dirichlet operator.
\end{abstract}
%
%
%
%
\section{Introduction}
We study about a non-stationary bulk queue arising in queuing theory. 
It is important to consider non-stationary property of queues in realistic queuing systems. 
In general, in realistic queuing systems such as welfare service system, first-aid system, repair shop and communication system, 
arrival rate of the customers for service varies with time by some factors. 
The stream of customers entering in service station such as stores, restaurants and barbershops varies according to time with 
period of a day or a week and the stream of ships entering in a port varies according to time with period of a year. 
In queueing theory, however, the stationary property of the streams by approximation to maximum or average streams have been 
assumed and the variability and periodicity of the streams have been disregarded in many cases. 
If a changeable stream of customers is approximated by maximum then idle time of servers increases. 
If it is approximated by average then it is unsuitable to correspond service rate with stream of customers. 
In the end the utility factor of systems decreases and the efficiency of service activity gets diminished.

For these reasons, the necessity to consider non-stationary queues has already been noticed and studied in many literatures 
concerned with queueing theory such as \cite{tho}. 
Since the stationary property of system parameters is destroyed in non-stationary queues, it is very difficult to study the specific properties of system. 
For the non-stationary queues in which analytical study is difficult, many approximate computation methods were proposed. 
In \cite{fel}, it was considered staffing of time-varying queues to achieve approximate time-stable performance. 
In \cite{fli}, it was obtained the queue length distribution of a multiple-server queuing system with time varying arrival and service 
rates when these rates were high.

In queueing theory, it is usually assumed that customers are arrived by one at a time. 
However, the queues with batch arrivals and batch services are more general and have several applications in telecommunications, 
manufacturing and computer systems. 
In modern wireless communication systems, especially when dealing with multimedia type of data, requests arrive in batches of 
varying sizes and services are provided in varying batches. 

In \cite{che} steady-state of a Markovian bulk-arrival and bulk-service queues was studied. 
In \cite{gup, haj} asymptotic stability of the solution of the $M(t) \vert$ $M[k, B] \vert 1$ queuing model with application of semi-group 
theory of operators was considered. 
In this paper, we consider the state probability equations of the non-stationary bulk queue $M(t) \vert$ $M[k, B] \vert 1$ and the abstract 
Cauchy problem to study behavior of state probability. 
In the queue $M(t) \vert$ $M[k, B] \vert 1$, services begin as soon as there are at least $k, k \geq 1$ customers in the queue.

%
%
%
%

\section{Structure of Non-stationary arrival bulk queue $M(t) \vert$ $M[k, B] \vert 1$}

We consider the queue with following structure. 
Customers are arrived according to non-stationary Poisson process of intensity $\lambda(t)$. 
The server starts service as soon as there are at least $k (k=1, 2, \cdots)$ customers in the queue. 
If a customer arrives while the server is busy, then the customer joins the queue. 
The server can at most serve $B$ customers simultaneously. 
The service time is exponentially distributed with parameter $\mu$. The system has only one server. 

We denote the queue with this structure by $M(t) \vert$ $M[k, B] \vert 1$. 
We need two time parameters to describe the non-stationary bulk queue $M(t) \vert$ $M[k, B] \vert 1$. 
The parameter $t \in [0, \infty)$ counts the time of the evolution of the whole system, whereas $x \in[0, \infty)$ counts the elapsed service time. 
The service time $x$ is reset to $0$ whenever a new service starts. 
Let $(n_1, n_2)$ denotes the states of the queue $M(t) \vert$ $M[k, B] \vert 1$ where $n_1$ represents number of customers in queue, 
$n_2=1$ represents that the server is busy and $n_2=0$ represents that the server is idle. 
We assume that $0 \leq r <k, ~ n \in \{0, 1, 2, \cdots \}, ~ t, x \geq 0$.

Let $p_{r,0}(t)$ denotes the probability that at time $t$ there are $r$ customers in queue to wait for service and the server is idle 
and $p_{r,1}(x, t)$ represents the density function of elapsed service time on condition that at time $t$ there are $n$ customers in queue. 
Then $\int_{(0, \infty)} p_{n, 1}(x, t)dx$ represents the probability that at time $t$ there are $n$ customers in queue to wait for service 
and the server is busy. For all $t \geq 0$ it follows that
\begin{equation*}
\sum_{r=0}^{k-1}p_{r,0}(t) + \sum_{n=0}^{\infty} p_{n, 1}(x, t)dx = 1.
\end{equation*}

%
%

\section{State probability equations}
\begin{thm}
For the queue $M(t) \vert$ $M[k, B] \vert 1$, following equations hold.
\begin{eqnarray}
&& \frac{dp_{0,0}(t)}{dt} = -\lambda(t) p_{0,0}(t) + \mu \int_{(0, \infty)}p_{0,1}(x, t)dx  \label{eq1} \\
&& \frac{dp_{r,0}(t)}{dt} = -\lambda(t) p_{r,0}(t) - \lambda(t) p_{r-1,0}(t) + \mu \int_{(0, \infty)}p_{r,1}(x, t)dx, ~ 0 \leq r<k  \qquad \label{eq2} 
\end{eqnarray}
\begin{eqnarray}
&& \frac{\partial p_{0,1}(x, t)}{\partial t} + \frac{\partial p_{0,1}(x, t)}{\partial x} = -(\lambda(t) + \mu)p_{0,1}(x, t)  \label{eq3} \\
&& \frac{\partial p_{n,1}(x, t)}{\partial t} + \frac{\partial p_{n,1}(x, t)}{\partial x} = -(\lambda(t) + \mu)p_{n,1}(x, t) + \lambda(t)p_{n-1,1}(x, t), ~ n \geq 1. \qquad  \label{eq4}
\end{eqnarray}
\end{thm}
\begin{proof}
First we derive equation \eqref{eq1}. 
Assume that at time $t$ there is no customer in queue and consider change of state in the interval $(t, t+\Delta t)$. 
The probability that the process goes to state $(0, 0)$ equals the sum of the probability that at time $t$ state is $(0, 0)$ and 
there is no customer arrived in the time interval, the probability that at time $t$ state is $(0, 1)$ and service for a customer 
end in the time interval, and the probability of the others. Since the probability of the others is $o(\Delta t)$, we have
\begin{equation*}
p_{0,0}(t+\Delta t) = p_{0,0}(t)(1-\lambda(t)\Delta t+o(\Delta t)) + (\mu \Delta t+o(\Delta t)) \int_0^{\infty}p_{0,1}(x, t)dx + o(\Delta t).
\end{equation*}

After ordering this equation and dividing by $\Delta t$, taking the limit as $\Delta t \to 0$ yields 
\begin{equation*}
\frac{dp_{0,0}(t)}{dt} = -\lambda(t)p_{0,0}(t) + \mu \int_0^{\infty}p_{0,1}(x,t)dx.
\end{equation*}

Equation \eqref{eq2} is derived similarly.

Next we derive equation \eqref{eq3}. The probability that at time $(t+\Delta t)$ the process goes to state $(0, 1)$ 
and elapsed service time is not greater than $(x+\Delta t)$ equals
\begin{equation*}
\int_{(0, x+\Delta t)}p_{0,1}(x, t+\Delta t)dx.
\end{equation*}
This probability equals the product of the probability that at time $t$ state is $(0, 1)$ and elapsed service time is not 
greater than $x$ and the probability that there is no customer arrived and ended service in the time interval $(t, t+\Delta t)$. 
That is, 
\begin{equation*}
\int_{(0, x+\Delta t)}p_{0,1}(x, t+\Delta t)dx = (1-\lambda(t)\Delta t + o(\Delta t))(1-\mu\Delta t + o(\Delta t)) \int_{(0, x)}p_{0,1}(x, t)dx.
\end{equation*}
From this expression we obtain
\begin{equation*}
p_{0,1}(x+\Delta t, t+\Delta t) - p_{0,1}(x, t) = -(\lambda(t)+\mu)p_{0,1}(x, t)\Delta t + o(\Delta t).
\end{equation*}
Therefore we have
\begin{equation*}
\frac{\partial p_{0,1}(x, t)}{\partial t} + \frac{\partial p_{0,1}(x, t)}{\partial x} = -(\lambda(t)+\mu)p_{0,1}(x, t).
\end{equation*}

Deriving of equation \eqref{eq4} is similar to the derivation of equation \eqref{eq3}. 
\end{proof}

For $x=0$ the following boundary conditions are imposed.
\begin{equation} \label{eq5}
\left\{
\begin{array}{ll}
p_{0,1}(0, t) = \mu \sum_{i=k}^B \int_{(0, \infty)}p_{i,1}(x, t)dx + \lambda(t)p_{k-1,0}(t) \\
p_{n,1}(0, t) = \mu \int_{(0, \infty)}p_{n+B,1}(x, t)dx, ~ n \geq 1
\end{array}
\right.
\end{equation}
As initial condition we assume that
\begin{equation} \label{eq6}
\left\{
\begin{array}{ll}
p_{0,0}(0) = 1 \\
p_{r,0}(0) = 0, ~ 1 \leq r \leq k-1 \\
p_{n,1}(x, 0) = 0, ~ n \geq 0
\end{array}
\right.
\end{equation}

%
%

\section{The abstract Cauchy problem}
To formulate the problem of state probability equations for the queue $M(t) \vert$ $M[k, B] \vert 1$ as an abstract Cauchy problem, 
we choose the state space as $X = \mathbf{C}^k \times l^1 (L^1[0, \infty))$.

For $\vec{p} = (p_{0,0}, \cdots, p_{k-1,0}, p_{0,1}(\cdot), p_{1,1}(\cdot), \cdots)^T \in X$, the norm of $\vec{p}$ is defined as follows.
\begin{equation*}
\Vert \vec{p} \Vert := \sum_{i=0}^{k-1}\vert p_{i,0} \vert + \sum_{i=0}^{\infty} \Vert p_{i,1}(\cdot)\Vert_{L^1[0, \infty)}.
\end{equation*} 

Define $(A_m, D(A_m))$, the operator on $X$, as follows.
\begin{eqnarray*}
&& A_m := \left(
\begin{array}{cc}
\mathbf{L} & \mathbf{M} \\
\mathbf{0} & \mathbf{K}
\end{array} \right), \\
&& D(A_m) := \mathbf{C}^k \times l^1 (W^{1,1}[0, \infty)),
\end{eqnarray*} 
where
\begin{equation*}
\mathbf{L}:= \left(
\begin{array}{c c c c c c}
-\lambda & 0 & 0 & \cdots & 0 & 0 \\
\lambda & -\lambda & 0 & \cdots & 0 & 0 \\
0 & \lambda & -\lambda & \cdots & 0 & 0 \\
\vdots & \vdots & \vdots & \ddots & \vdots & \vdots \\
0 & 0 & 0 & \cdots & \lambda & -\lambda
\end{array} \right) 
\end{equation*} 
has a format of $k$ order matrix.
\begin{equation*}
\mathbf{M}:= \left(
\begin{array}{c c c c c c c}
\mu\psi & 0 & 0 & \cdots & 0 & 0 & \cdots \\
0 & \mu\psi & 0 & \cdots & 0 & 0 & \cdots \\
\vdots & \vdots & \vdots & \ddots & \vdots & \vdots & \ddots \\
0 & 0 & 0 & \cdots & \mu\psi & 0 & \cdots
\end{array} \right) 
\end{equation*}  
and
\begin{equation*}
\mathbf{K}:= \left(
\begin{array}{c c c c}
\mathbf{D} & 0 & 0 & \cdots \\
\lambda & \mathbf{D} & 0 & \cdots \\
0 & \lambda & \mathbf{D} & \cdots \\
\vdots & \vdots & \vdots & \ddots 
\end{array} \right).
\end{equation*}  
The operator $\psi$ is defined as follows.
\begin{equation*}
\psi : L^1[0, \infty) \to \mathbf{C}, \qquad f \mapsto \psi(f):=\int_{(0, \infty)}f(x)dx
\end{equation*} 
\begin{equation*}
\mathbf{D}f := -\frac{d}{dx}f - (\lambda+\mu)f.
\end{equation*}  
Clearly the operator $(A_m, D(A_m))$ is a closed operator.

As boundary space we choose 
\begin{equation*}
\partial X:=l^1
\end{equation*}  
and define the boundary operators as follows.
\begin{equation*}
L : D(A_m) \to \partial X,
\end{equation*}  
\begin{equation*}
(p_{0,0}, \cdots, p_{k-1,0}, p_{0,1}(\cdot), p_{1,1}(\cdot), \cdots)^T \mapsto (p_{0,1}(0), p_{1,1}(0), \cdots)^T.
\end{equation*}  
  
By operator matrix
\begin{equation*}
\Phi = \left(
\begin{array}{c c c c c c c c c c c c c c}
0 & \cdots & 0 & \lambda & 0 & \cdots & 0 & \mu\psi & \cdots & \mu\psi & 0 & 0 & 0 & \cdots \\
0 & \cdots & 0 & 0 & 0 & \cdots & 0 & 0 & \cdots & 0 & \mu\psi & 0 & 0 & \cdots \\
0 & \cdots & 0 & 0 & 0 & \cdots & 0 & 0 & \cdots & 0 & 0 & \mu\psi & 0 & \cdots \\
0 & \cdots & 0 & 0 & 0 & \cdots & 0 & 0 & \cdots & 0 & 0 & 0 & \mu\psi & \cdots \\
\vdots & \ddots & \vdots & \vdots & \vdots & \ddots & \vdots & \vdots & \ddots & \vdots & \vdots & \vdots & \vdots& \ddots
\end{array} \right)
\end{equation*}  
we define the operator
\begin{equation*}
\Phi : D(A_m) \to \partial X.
\end{equation*}  
In the operator matrix $\Phi$, there are $k-1$ zeroes in front of $\lambda$ at the first row, $k$ zeroes between the $\lambda$ and $\mu\psi$. 
There are $B-k+1 ~ \mu\psi$'s at the row.

Now define $(A, D(A))$, the operator on $X$, as follows.
\begin{eqnarray*}
&& A\vec{p}:=A_m\vec{p}, \\
&& D(A):=\left\{\vec{p} \in D(A_m) \vert L\vec{p} = \Phi \vec{p} \right\}.
\end{eqnarray*}  
 
Using these definitions, \eqref{eq1} $\sim$ \eqref{eq6}, the problem of state probability equations for the queue 
$M(t) \vert$ $M[k, B] \vert 1$, is formulated as the following abstract Cauchy problem 
\begin{equation*}
\left\{
\begin{array}{ll}
\frac{d\vec{p}(t)}{dt} = A\vec{p}(t), ~ t \in [0, \infty) \\
\vec{p}(0) = (1, 0, 0, \cdots)^T \in X
\end{array}
\right..
\end{equation*}

%
%

\section{The eigenfunction of maximal operator}
The following abbreviations are used in the sequel:
\begin{equation*}
\Gamma := \gamma+\lambda+\mu, ~ \Lambda := \gamma+\lambda.
\end{equation*}
\begin{thm}
Let $S = \left\{ \gamma \in C \vert \text{Re}\gamma > -\mu, \gamma \neq -\lambda(t) \right\}$. 
For $\gamma \in S$,
\begin{equation*}
\vec{p} = (p_{0,0}, \cdots, p_{k-1,0}, p_{0,1}(\cdot), p_{1,1}(\cdot), p_{2,1}(\cdot), \cdots)^T \in \textnormal{Ker}(\gamma I-A_m)
\end{equation*}
holds true if and only if there exists $(c_n)_{n \geq 1} \in l^1$ such that
\begin{eqnarray}
p_{0,0} & = & \frac{\mu c_1}{\Gamma(t)\Lambda(t)} \label{eq7} \\
p_{r,0} & = & \frac{1}{\Lambda(t)} \left( \lambda p_{r-1,0} + \mu \sum_{i=1}^{r+1}c_i \frac{\lambda(t)^{r+1-i}}{\Gamma^{r+2-i}} \right), 
~ 1 \leq r \leq k-1 \label{eq8} \\
p_{n,1}(x) & = & e^{-\Gamma(x)} \sum_{i=1}^{n+1}c_i \frac{\lambda(t)^{n+1-i}}{(n+1-i)!} x^{n+1-i}, ~ n \geq 0. \label{eq9}
\end{eqnarray} 
\end{thm}
\begin{proof}
We first verify that each $\vec{p}$ given as in \eqref{eq7} $\sim$ \eqref{eq9} is contained in $D(A_m)$. 
Note that for $0 \neq c \in \mathbf{C}$ and $i \in \mathbf{N}$,
\begin{equation*}
\int_0^{\infty}e^{-cx}x^idx = \frac{i!}{c^{i+1}}.
\end{equation*} 
Using this we estimate the norm. 
\begin{eqnarray*}
\Vert p_{n,1}(\cdot) \Vert_{L^1[0, \infty)} & = & \int_0^{\infty} \left\vert e^{-\Gamma x} \sum_{i=1}^{n+1}c_i \frac{\lambda(t)^{n+1-i}}{(n+1-i)!}x^{n+1-i} \right\vert dx \\
& \leq & \sum_{i=1}^{n+1} \vert c_i\vert \frac{\lambda(t)^{n+1-i}}{(n+1-i)!} \frac{(n+1-i)!}{(\text{Re}\Gamma)^{n+2-i}} \\
& = & \sum_{i=0}^{n} \vert c_{n+1-i}\vert \frac{\lambda(t)^i}{(\text{Re}\Gamma)^{i+1}}.
\end{eqnarray*} 
Since $\text{Re}\gamma > -\mu$, the series
\begin{equation*}
\sum_{i=0}^{n} \left( \frac{\lambda(t)}{\text{Re}\Gamma} \right)^i
\end{equation*}
converges absolutely. Therefore 
\begin{eqnarray*}
\sum_{n=0}^{\infty} \Vert p_{n,1}(\cdot) \Vert_{L^1[0, \infty)} & \leq & \sum_{n=0}^{\infty}\sum_{i=0}^n \vert c_{n+1-i}\vert \frac{1}{\text{Re}\Gamma} \left( \frac{\lambda(t)}{\text{Re}\Gamma} \right)^i \\
& = & \frac{1}{\text{Re}\Gamma} \sum_{i=0}^{\infty}\left( \frac{\lambda(t)}{\text{Re}\Gamma} \right)^i \left(\sum_{i=1}^{\infty} \vert c_i \vert \right) \\
& = & \frac{1}{\text{Re}\Gamma} \frac{1}{1-\frac{\lambda(t)}{\text{Re}\Gamma}} \Vert (c_i)_{i\geq 1} \Vert_{l^1} <\infty.
\end{eqnarray*}  
Hence, the norm $\Vert \vec{p} \Vert_{D(A_m)}$ of $\vec{p}$ is finite and $\vec{p} \in D(A_m)$. 
And we can easily verify that each $\vec{p}$ as in \eqref{eq7} $\sim$ \eqref{eq9} satisfies 
\begin{equation*}
(\gamma I - A_m)\vec{p} = 0.
\end{equation*}

Conversely, we assume that $\vec{p} \in \text{Ker}(\gamma I - A_m)$. Then we get a system of differential equations from
\begin{equation*}
(\gamma I - A_m)\vec{p} = 0.
\end{equation*}
Solving this we immediately get \eqref{eq7} $\sim$ \eqref{eq9}. From
\begin{equation*}
\sum_{i=1}^{\infty} \vert c_i \vert = \sum_{i=1}^{\infty} \vert p_{i,1}(0)\vert \leq \sum_{i=1}^{\infty} \Vert p_{i,1}(\cdot)\Vert_{W^{1,1}[0, \infty)} \leq \Vert \vec{p} \Vert_{D(A_m)}<\infty,
\end{equation*}
we obtain that $(c_n)_{n\geq 1} \in l^1$.    
\end{proof}

%
%

\section{A property of the Dirichlet operator}
We consider the following operator $A_0$.
\begin{equation*}
A_0 \vec{p} = A_m\vec{p}
\end{equation*} 
\begin{equation*}
D(A_0) = \{\vec{p} \in D(A_m) \vert L\vec{p} = 0\}
\end{equation*} 
Since the boundary operator $L$ is surjective, if $\gamma \in \rho(A_0)$ then 
\begin{equation*}
\left. L \right\vert_{\text{Ker}(\gamma I-A_m)} : \text{Ker}(\gamma I-A_m) \to \partial X
\end{equation*} 
is invertible. Now consider the Dirichlet operator $D_{\gamma}$, its inverse. For $i \in \mathbf{N}$ we define the 
operator $\varepsilon_i : \mathbf{C} \to L^1[0, \infty)$ as follows.
\begin{equation*}
(\varepsilon_i(c))(x) = c\frac{\lambda(t)^i}{i!}x^ie^{-\Gamma x}, ~ c \in \mathbf{C}, x \in [0, \infty).
\end{equation*} 
If $\gamma \in \mathbf{C}, \text{Re}\gamma>-\mu, ~ \gamma \neq -\lambda(t)$ then the following equation holds.
\begin{equation*}
D_{\gamma} = \left(
\begin{array}{c c c c c c c}
d_{1,1} & 0 & 0 & \cdots & 0 & 0 & \cdots \\
d_{2,1} & d_{2,2} & 0 & \cdots & 0 & 0 & \cdots \\
\vdots & \vdots  & \vdots  & \ddots & \vdots  & \vdots  & \ddots \\
d_{k,1} & d_{k,2} & d_{k,3} & \cdots & d_{k,k} & 0 & \cdots \\
\varepsilon_0 & 0 & 0 & \cdots & 0 & 0 & \cdots \\
\varepsilon_1 & \varepsilon_0 & 0 & \cdots & 0 & 0 & \cdots \\
\varepsilon_2 & \varepsilon_1 & \varepsilon_0 & \cdots & 0 & 0 & \cdots \\
\vdots & \vdots  & \vdots  & \ddots & \vdots  & \vdots  & \ddots 
\end{array} \right),
\end{equation*}  
where
\begin{equation*}
d_{i,r} = \frac{\mu\lambda^{i+1-r}}{\Gamma\Lambda^{i+2}}\sum_{j=0}^{i+1-r}\frac{\Lambda^{r+j}}{\Gamma^j}.
\end{equation*} 

To obtain spectrum of the operator $A$, we need expression of $\Phi D_{\gamma}$. We have
\begin{equation*}
\Phi D_{\gamma} = \left(
\begin{array}{c c c c c c c c}
a_{1,1} & a_{1,2} & \cdots & a_{1,1} & 0 & 0 & 0 & \cdots \\
\frac{\mu}{\Gamma}\left(\frac{\lambda}{\Gamma}\right)^{B+1} & \frac{\mu}{\Gamma}\left(\frac{\lambda}{\Gamma}\right)^{B} & \cdots & \frac{\mu}{\Gamma} \frac{\lambda}{\Gamma} & \frac{\mu}{\Gamma} & 0 & 0 & \cdots \\
\frac{\mu}{\Gamma}\left(\frac{\lambda}{\Gamma}\right)^{B+2} & \frac{\mu}{\Gamma}\left(\frac{\lambda}{\Gamma}\right)^{B+1} & \cdots & \frac{\mu}{\Gamma}\left(\frac{\lambda}{\Gamma}\right)^2 & \frac{\mu}{\Gamma} \frac{\lambda}{\Gamma} & \frac{\mu}{\Gamma} & 0 & \cdots \\
\vdots & \vdots & \ddots & \vdots & \vdots & \vdots & \vdots & \ddots
\end{array} \right),
\end{equation*}  
where
\begin{equation*}
a_{1,i} = \frac{\mu}{\Gamma}\left(\frac{\lambda}{\Gamma}\right)^{k+1-i} \sum_{j=0}^{k-i}\left(\frac{\Lambda}{\Gamma}\right)^j + 
\frac{\mu}{\Gamma} \sum_{j=k+1-i}^{B+1-i}\left(\frac{\lambda}{\Gamma}\right)^j, ~ 1 \leq i \leq k,
\end{equation*} 
\begin{equation*}
a_{1,i} = \frac{\mu}{\Gamma} \sum_{j=0}^{B+1-i} \left(\frac{\lambda}{\Gamma}\right)^j, ~ k+1 \leq i \leq B+1.\\
\end{equation*}

%

 \textbf{Acknowledgement} I would like to thank anonymous referees for their valuable comments and suggestion.

 \end{document}